\documentclass[11pt,lenq]{amsart}
\usepackage{amsmath}
\usepackage{amscd}
\usepackage{amssymb}
\usepackage{amsbsy}
\usepackage{amsfonts}
\usepackage{color}
\usepackage{latexsym}
\usepackage{graphics}
\usepackage{graphicx}
\usepackage{amsmath,amscd,latexsym}
\usepackage{array}
\usepackage{enumerate}
\theoremstyle{plain}
\newtheorem{theorem}{Theorem}[section]
\newtheorem{proposition}[theorem]{Proposition}
\newtheorem{lemma}[theorem]{Lemma}
\newtheorem{corollary}[theorem]{Corollary}

\newtheorem{definition}[theorem]{Definition}

\newtheorem{example}[theorem]{Example}
\newtheorem{main theorem}[theorem]{Main Theorem}

\newlength\savewidth

\newcommand{\svert}{\,|\,}

\newcommand{\im}{{\rm im}}
\newcommand{\Lab}{{\bf Lab}}

\begin{document}
\title[A Recipe for Constructing Non-Hopfian Relatively Hyperbolic Groups with Hopfian Peripheral Subgroups]
{A Recipe for Constructing Non-Hopfian Relatively Hyperbolic Groups with Hopfian Peripheral Subgroups}

\author{Jan Kim}
\address{Department of Mathematics\\
Pusan National University \\
San-30 Jangjeon-Dong, Geumjung-Gu, Pusan, 609-735, Korea}
\email{jankim@pusan.ac.kr}

\author{Tattybubu Arap kyzy}
\address{Department of Mathematics\\
Pusan National University \\
San-30 Jangjeon-Dong, Geumjung-Gu, Pusan, 609-735, Korea}
\email{tattuu\_arapova@pusan.ac.kr}

\author{Donghi Lee}
\address{Department of Mathematics\\
Pusan National University \\
San-30 Jangjeon-Dong, Geumjung-Gu, Pusan, 609-735, Korea}
\email{donghi@pusan.ac.kr}

\subjclass[2020]{{Primary 20F65, 20F06}\\
\indent {The third author was supported by Basic Science Research Program
through the National Research Foundation of Korea(NRF) funded
by the Ministry of Education, Science and Technology(2020R1F1A1A01071067).}}

\begin{abstract}
Very recently, Kim and Lee~\cite{Kim_Lee} presented an example of
a non-Hopfian relatively hyperbolic group with a Hopfian peripheral subgroup,
demonstrating a counterexample to Osin's well-known question~\cite[Problem~5.5]{Osin}.
In this paper, we provide a general construction method using the so-called image extension theorem
to generate non-Hopfian relatively hyperbolic groups with Hopfian peripheral subgroups.
Additionally, we provide two specific examples using this construction method.
\end{abstract}

\maketitle

\section{Introduction}

In 1987, Gromov~\cite{Gromov} introduced the concept of hyperbolic groups
and proposed relatively hyperbolic groups as a generalization of hyperbolic groups.
Subsequently, several authors provided alternative characterizations of relatively hyperbolic groups
(see \cite{Bowditch, Farb, Groves_Manning, Osin}).
When a group $G$ is hyperbolic relative to a collection of subgroups $\mathbb{H}={ H_\lambda}_{\lambda \in \Lambda}$, we refer to $\mathbb{H}$ as a {\it peripheral structure} of $G$,
and any element in $\mathbb{H}$ is referred to as a {\it peripheral subgroup} of $G$,

Recall that a group $G$ is {\it Hopfian} if every epimorphism $G \rightarrow G$ is an automorphism.
Sela proved in his work \cite{Sela} that every torsion-free hyperbolic group is Hopfian.
This result was later extended by Reinfeldt and Weidmann \cite{Reinfeldt_Weidmann}
to encompass any hyperbolic group.
Recently, Fujiwara and Sela \cite{Fujiwara_Sela} presented an alternative proof for
the Hopf property of hyperbolic groups.

In relation to the Hopf property for relatively hyperbolic groups,
Osin~\cite{Osin} posed a question of
whether all relatively hyperbolic groups remain Hopfian
when their peripheral subgroups are Hopfian.
In a recent development, Kim and Lee~\cite{Kim_Lee}
provided a negative answer to this question by constructing
a non-Hopfian relatively hyperbolic group
with a Hopfian peripheral subgroup as follows.

\begin{theorem}[{\cite[Theorem~1.1]{Kim_Lee}}]
\label{thm:the_first_example}
Let $H_0$ be the group given by the presentation
\[
H_0=\langle b, c \svert b^2=c^9=1, \ b^{-1}cb=c^{-1} \rangle,	
\]
and take successively two HNN-extensions from $H_0$ as follows:
\[
H_1=\langle H_0, s \svert s^{-1}bs=bc^{-3}, \ s^{-1}cs=c \rangle;
\]
\begin{equation}
\label{equ:H_2_presentation}
H_2=\langle H_1, t \svert t^{-1}st=s^3 \rangle.
\end{equation}
Next, form the free product $H=H_2 \ast \langle e, f \svert \emptyset \rangle$.
Finally, letting $\langle a \rangle$ be an infinite cyclic group,
take successively two multiple HNN-extensions from $H \ast \langle a \rangle$ as follows:
\begin{subequations}
\begin{align}
\label{equ:K_presentation}
K&=\langle H \ast \langle a \rangle, u, v \svert u^{-1}(bacb^{-1})u=a, \  v^{-1}av=tst^{-1} \rangle;\\
\label{equ:G_presentation}
G&=\langle K, x, y \svert x^{-1}ux=c^3ec^3e^{-1}, \ y^{-1}vy=c^3fc^3f^{-1} \rangle.		
\end{align}	
\end{subequations}
Then $G$ is a non-Hopfian group which is hyperbolic relative to the Hopfian subgroup $H$.
\end{theorem}

This construction may seem coincidental,
but it actually encompasses a hidden logic,
which we present in Theorem~\ref{thm:image_extension_theorem}.
For the definition of a relative presentation,
we refer the reader to \cite{Bogley_Pride}.
We use the notation $[a, b]$ to denote the commutator $aba^{-1}b^{-1}$
for any two elements $a$ and $b$ in a group.

\begin{theorem}[Image Extension Theorem]
\label{thm:image_extension_theorem}
Let $\psi$ be an endomorphism of a group $\mathbf{H}$.
Suppose that $u$ and $v$ are nontrivial elements of $\mathbf{H}$
such that $u \in \ker \psi \cap \im \psi$ and $v \in \ker \psi \setminus \im \psi$.
Suppose further that $u^2 \neq 1$ in $\mathbf{H}$ and that $v$ has infinite order.
Now define a group $\mathbf{G}$ by the relative presentation
\begin{equation}
\label{equ:image_extension_theorem_G}
\mathbf{G}=\langle \mathbf{H}, x, t \svert t^{-1}vt=[u,x] \rangle.
\end{equation}
Also define an endomorphism $\tilde{\psi}$ of $\mathbf{G}$ such that
\[
\tilde{\psi}|_\mathbf{H}=\psi, \quad \tilde{\psi}(x)=x \quad \textrm{and} \quad \tilde{\psi}(t)=t.
\]
Then the following hold.
\begin{enumerate}[\indent \rm (i)]
\item $\mathbf{G}$ contains $\mathbf{H}$ as a subgroup.

\item $\mathbf{G}$ is hyperbolic relative to $\mathbf{H}$.

\item $\tilde{\psi}$ is an extension of $\psi$ such that $v \in \im \tilde{\psi}$.
\end{enumerate}	
\end{theorem}

Here, it is evident that $\tilde{\psi}$ is not injective,
as the non-identity element $u$ belongs to the kernel of $\tilde{\psi}$.
Therefore, the following conclusion follows immediately.

\begin{corollary}
\label{cor:image_extension_theorem}
In Theorem~\ref{thm:image_extension_theorem},
suppose further that $\mathbf{H}$ is Hopfian and $\tilde{\psi}$ is surjective.
Then $\mathbf{G}$ is a non-Hopfian relatively hyperbolic group with respect to a Hopfian subgroup $\mathbf{H}$.
\end{corollary}

We adopt this corollary as a general recipe to produce
a non-Hopfian relatively hyperbolic group with respect to a Hopfian subgroup.
Additionally, we provide two specific examples
(see, for the proofs, Propositions~\ref{prop:first_example} and \ref{prop:second_example})
to illustrate the application of Corollary~\ref{cor:image_extension_theorem}.
In both examples, we employ groups, as peripheral subgroups, that were previously established as Hopfian~\cite{Kim_Lee, Andreadakis}.

\begin{example}
\label{example:first_example}
Let $\mathbf{H}$ be a group given by the presentation
\[
\begin{aligned}
\mathbf{H}=\langle b, c, s, k \svert
&b^2=c^9=1, \ b^{-1}cb=c^{-1},\\
&s^{-1}bs=bc^{-3}, \ s^{-1}cs=c,\\
&k^{-1}sk=s^3
\rangle,
\end{aligned}
\]
and suppose that $\mathbf{G}$ is a group given by the relative presentation
\[
\mathbf{G}=\langle \mathbf{H}, x, t \svert  t^{-1}\{(ks^{-1}k^{-1})b(ksk^{-1})cb^{-1}\}t= [c^3,x]  \rangle.
\]
Then $\mathbf{G}$ is a non-Hopfian relatively hyperbolic group with respect to a Hopfian subgroup $\mathbf{H}$.		
\end{example}

\begin{example}
\label{example:second_example}
Let $\mathbf{H}$ be a group given by the presentation
\[
\mathbf{H}=\langle a,b,s \svert [a,b]=1,\ s^{-1}a^2s=a^4 \rangle,
\]
and suppose that $\mathbf{G}$ is a group given by the relative presentation
\[
\mathbf{G}=\langle \mathbf{H}, x, t \svert  t^{-1}(s^{-1}asa^{-2})t= [[sa^2s^{-1},b],x] \rangle.
\]	
Then $\mathbf{G}$ is a non-Hopfian relatively hyperbolic group with respect to a subgroup $\mathbf{H}$.
\end{example}

The present paper is organized as follows.
In Section~\ref{sec:necessary_definitions_and_theorems},
we review necessary definitions and theorems
that will be used in the subsequent sections.
Section~\ref{sec:proof_of_the_image_extension_theorem}
is dedicated entirely to proving Theorem~\ref{thm:image_extension_theorem}.
In Lemma~\ref{lem:H_subgroup_of_G},
we establish the proof for part (i) of Theorem~\ref{thm:image_extension_theorem}.
To demonstrate part (ii) of Theorem~\ref{thm:image_extension_theorem},
we first establish that
$\mathbf{H} \ast \langle x \rangle$ is hyperbolic relative to the collection of subgroups
$\{\mathbf{H}, \langle [u,x] \rangle\}$
by using Lemmas~\ref{lem:hyperbolic element} and \ref{lem:unique_elementary_subgroup}
together with Osin's theorem about hyperbolically embedded subgroups.
We then employ Osin's combination theorem for relatively hyperbolic groups
to deduce that $\mathbf{G}$ is hyperbolic relative to $\mathbf{H}$. 
Part (iii) of Theorem~\ref{thm:image_extension_theorem} easily follows from
the hypothesis that $u \in \im \psi$.
Finally, in Section~\ref{sec:proof_proposition}, we rename Examples~\ref{example:first_example} and \ref{example:second_example} as Propositions~\ref{prop:first_example} and \ref{prop:second_example}, respectively, and present their respective proofs.

\section{Necessary definitions and theorems}
\label{sec:necessary_definitions_and_theorems}

In this section,
we review necessary definitions and theorems
that will be used to prove Theorem~\ref{thm:image_extension_theorem}.

\subsection{Relatively hyperbolic groups}

The following definition and theorem about hyperbolically embedded subgroups
of relatively hyperbolic groups are attributed to Osin~\cite{Osin3}.

\begin{definition}[Hyperbolically embedded subgroups]
{\rm Let $G$ be hyperbolic relative to a collection of subgroups $\mathbb{H}$.
A subgroup $Q \le G$ is said to be {\em hyperbolically embedded into} $G$,
if $G$ is hyperbolic relative to $\mathbb{H} \cup \{Q\}$.
}
\end{definition}

Recall that an element in a relatively hyperbolic group $G$ is referred to as {\it hyperbolic}
if it has infinite order and not conjugate to any element of a peripheral subgroup of $G$.

\begin{theorem}[{\cite[Theorem~4.3, Corollary~1.7]{Osin3}}]
\label{thm:hyperbolically_embedded}
Let $G$ be hyperbolic relative to a collection of subgroups $\mathbb{H}$.
Then for any hyperbolic element $g \in G$, the unique maximal elementary subgroup $E(g)$
containing $g$ is hyperbolically embedded into $G$, where $E(g)$ is the set
\[
E(g)=\{f \in G : f^{-1}g^nf = g^{\pm n} \ \textrm{for some} \ n \in \mathbb{N}\}.
\]	
\end{theorem}

We recall one of Osin's combination theorems for relatively hyperbolic groups.

\begin{theorem}[{\cite[Corollary~1.4]{Osin2}}]\label{thm:Osin_combination_theorem}
Suppose that a group $G$ is hyperbolic relative to a collection of subgroups
$\mathbb{H}=\{ H_\lambda \}_{\lambda \in \Lambda}$.
Assume in addition that there exists a monomorphism $\iota : H_\mu \rightarrow H_\nu$ for some $\mu \neq \nu \in \Lambda$,
and that $H_\mu$ is finitely generated.
Then the HNN-extension
\[
G^*=\langle G, t \svert t^{-1}ht=\iota(h), h \in H_\mu \rangle
\]
is hyperbolic relative to the collection $\mathbb{H} \setminus \{H_\lambda\}$.		
\end{theorem}

\subsection{Van Kampen diagrams and annular diagrams}

The definitions and notation for van Kampen diagrams and annular diagrams
follow Lyndon and Schupp~\cite[Chapter~V.1]{Lyndon_Schupp}.
The notation $(u)$ represents the cyclic word associated with
a cyclically reduced word $u$, and
$(u) \equiv (v)$ indicates the visual equality of two cyclic words $(u)$ and $(v)$,
meaning that $v$ is a cyclic shift of $u$.

The following two well-known theorems serve as crucial tools in proving Lemma~\ref{lem:unique_elementary_subgroup}.

\begin{theorem}[van Kampen's Lemma]
\label{thm:van_Kampen}
Suppose $G=\langle X \svert R \rangle$, where every $r \in R$ is cyclically reduced.
Let $u$ be a cyclically reduced word in $X$
which represents the identity element in $G$
if and only if
there exists a reduced van Kampen diagram $\Delta$ over $G=\langle X \svert R \rangle$
such that
$(\Lab(\partial \Delta)) \equiv (u)$, where $\Lab$ is a labeling function.				
\end{theorem}

\begin{theorem}[Schupp's Lemma]
\label{thm:schupps_lemma}
Suppose $G=\langle X \svert R \rangle$, where every $r \in R$ is cyclically reduced.
Let $u, v$ be two cyclically reduced words in $X$
which are nontrivial in $G$.
Then $u$ and $v$ represent conjugate elements in $G$ if and only if
there exists a reduced annular diagram $\Delta$ over $G=\langle X \svert R \rangle$
such that
$(\Lab(\partial_\textup{out} \Delta)) \equiv (u)$ and $(\Lab(\partial_\textup{inn} \Delta)) \equiv (v^{-1})$.
Here, $\partial_\textup{out}\Delta$ is the boundary of the unbounded component of $\mathbb{R}^2 \setminus \Delta$ and $\partial_\textup{inn}\Delta$ is the boundary of the bounded component of $\mathbb{R}^2 \setminus \Delta$.	
\end{theorem}

As is customary, we assign orientations to all van Kampen diagrams
such that their boundaries are read clockwise.
Accordingly, the orientations for annular diagrams are chosen
so that their outer boundaries are read clockwise,
while the inner boundaries are read counterclockwise.

\section{Proof of Theorem~\ref{thm:image_extension_theorem}}
\label{sec:proof_of_the_image_extension_theorem}

Throughout this section, we use $\mathbf{H}$ and $\mathbf{G}$
to denote the groups as stated in Theorem~\ref{thm:image_extension_theorem}.
The following lemma provides a proof for part (i) of Theorem~\ref{thm:image_extension_theorem}.

\begin{lemma}
\label{lem:H_subgroup_of_G}
The group $\mathbf{G}$ contains $\mathbf{H}$ as a subgroup.
\end{lemma}

\begin{proof}
By the hypothesis of Theorem~\ref{thm:image_extension_theorem},
the element $v$ of $\mathbf{H} \ast \langle x \rangle$ has infinite order.
The element $[u,x]$ of $\mathbf{H} \ast \langle x \rangle$ also has infinite order,
since $[u,x]^m \neq 1$ in $\mathbf{H} \ast \langle x \rangle$
for any $m \in \mathbb{Z} \setminus \{0\}$ by the normal form theorem for free products.
So there is an isomorphism $\iota$ between two
infinite cyclic subgroups $\langle v \rangle$ and $\langle [u,x] \rangle$
of $\mathbf{H} \ast \langle x \rangle$
such that $\iota(v)=[u,x]$.
This yields that
$\mathbf{G}$ can be seen as an HNN-extension from $\mathbf{H} \ast \langle x \rangle$
with associated isomorphism $\iota$ as follows:
\[
\mathbf{G}=\langle \mathbf{H} \ast \langle x \rangle, t \svert t^{-1}vt=\iota(v) \rangle.
\]
Thus $\mathbf{G}$ contains $\mathbf{H} \ast \langle x \rangle$ as a subgroup.
Since $\mathbf{H} \ast \langle x \rangle$ clearly contains $\mathbf{H}$ as a subgroup,
$\mathbf{G}$ contains $\mathbf{H}$ as a subgroup, as desired.
\end{proof}

To establish part (ii) of Theorem~\ref{thm:image_extension_theorem},
we first prove the following two lemmas, which will play a crucial role in the proof.
It is evident that the free product $\mathbf{H} \ast \langle x \rangle$
is hyperbolic relative to the collection of subgroups $\{\mathbf{H}, \langle x \rangle \}$.
Since $\langle x \rangle$ is an infinite cyclic group which is hyperbolic,
$\mathbf{H} \ast \langle x \rangle$ becomes hyperbolic relative to $\mathbf{H}$.
Taking this perspective into account, we prove the following

\begin{lemma}
\label{lem:hyperbolic element}
The commutator $[u,x]$ is a hyperbolic element in $\mathbf{H} \ast \langle x \rangle$
which is regarded as a relatively hyperbolic group with respect to the subgroup $\mathbf{H}$.
\end{lemma}

\begin{proof}
By the proof of Lemma~\ref{lem:H_subgroup_of_G},
$[u,x]$ is an element of $\mathbf{H} \ast \langle x \rangle$ of infinite order.
To prove that $[u,x]$ is a hyperbolic element in $\mathbf{H} \ast \langle x \rangle$,
it suffices to show that $[u,x]$ is not conjugate to
any element of $\mathbf{H}$.
Assume on the contrary that there exists
an element $h$ of $\mathbf{H}$ such that
$h$ is conjugate to $[u,x]$ in $\mathbf{H} \ast \langle x \rangle$.
But this is impossible,
since every element conjugate to $[u,x]$ in $\mathbf{H} \ast \langle x \rangle$
is a cyclic permutation of $[u,x]$ by Conjugacy Theorem for free products
~\cite[Theorem~4.2]{Magnus_Karrass_Solitar}.
\end{proof}

\begin{lemma}
\label{lem:unique_elementary_subgroup}
The unique maximal elementary subgroup $E([u,x])$ of $\mathbf{H} \ast \langle x \rangle$
is precisely the cyclic subgroup $\langle [u,x] \rangle$.
\end{lemma}

\begin{proof}
Suppose that $f$ is an element of $E([u,x])$,
that is, $[u,x]^n=f[u,x]^{\pm n}f^{-1}$ for some $n \in \mathbb{N}$.
Then by Theorem~\ref{thm:schupps_lemma}, there exist a reduced annular diagram $\Delta$ over
$\mathbf{H} \ast \langle x \rangle$ and a path $p$ lying in $\Delta$ such that
\begin{enumerate}[\indent \rm (i)]
	\item $p_{-} \in \partial_\text{out} \Delta$, where $p_{-}$ denotes the origin of $p$;
	\item $p_{+} \in \partial_\text{inn} \Delta$, where $p_{+}$ denotes the terminus of $p$;
	\item $\Lab(\partial_\text{out} \Delta|_{p_{-}})\equiv [u,x]^n$;
	\item $\Lab(\partial_\text{inn} \Delta|_{p_{+}}) \equiv [u,x]^{\mp n}$;
	\item $\Lab(p) \equiv f$.	
\end{enumerate}

In particular, we choose $\Delta$ to be {\em minimal},
meaning that it contains the fewest number of $2$-cells among all diagrams satisfying
conditions (i)-(v).
Since there are no defining relators involving the letter $x^{\pm 1}$ in $\mathbf{H} \ast \langle x \rangle$, there cannot be any $2$-cell, say $\pi$, in $\Delta$ such that $\partial \pi \cap \partial_\text{out}\Delta$ contains an edge labeled $x^{\pm 1}$.
Hence, every edge labeled $x^{\pm 1}$ in $\partial_\text{out}\Delta$
must overlap with an edge labeled $x^{\mp 1}$, respectively, in $\partial_\text{inn}\Delta$.

We consider two cases separately, depending on whether
$\Lab(\partial_\text{inn} \Delta|_{p_{+}}) \equiv [u,x]^n$
or $\Lab(\partial_\text{inn} \Delta|_{p_{+}}) \equiv [u,x]^{-n}$.

\medskip
\noindent{\bf Case~1.} $\Lab(\partial_\text{out} \Delta|_{p_{-}})\equiv [u,x]^n $ and
$\Lab(\partial_\text{inn} \Delta|_{p_{+}}) \equiv [u,x]^n$.

\medskip
Recalling that the outer boundary $\partial_\text{out}\Delta$ is read clockwise,
while the inner boundary $\partial_\text{inn}\Delta$ is read counterclockwise,
we see that $\Delta$ can only have the shape depicted in Figure~\ref{fig:annular_diagram_2}.
As depicted in this figure, $\Delta$ contains a circular subdiagram with a boundary label $(u^2)$. According to Theorem~\ref{thm:van_Kampen},
this implies that $u^2=1$ in $\mathbf{H} \ast \langle x \rangle$,
which contradicts the hypothesis of Theorem~\ref{thm:image_extension_theorem}
stating that $u^2 \neq 1$ in $\mathbf{H}$. Therefore, this case cannot occur.

\begin{figure}[h]
\begin{center}
\includegraphics[width=0.5\columnwidth]{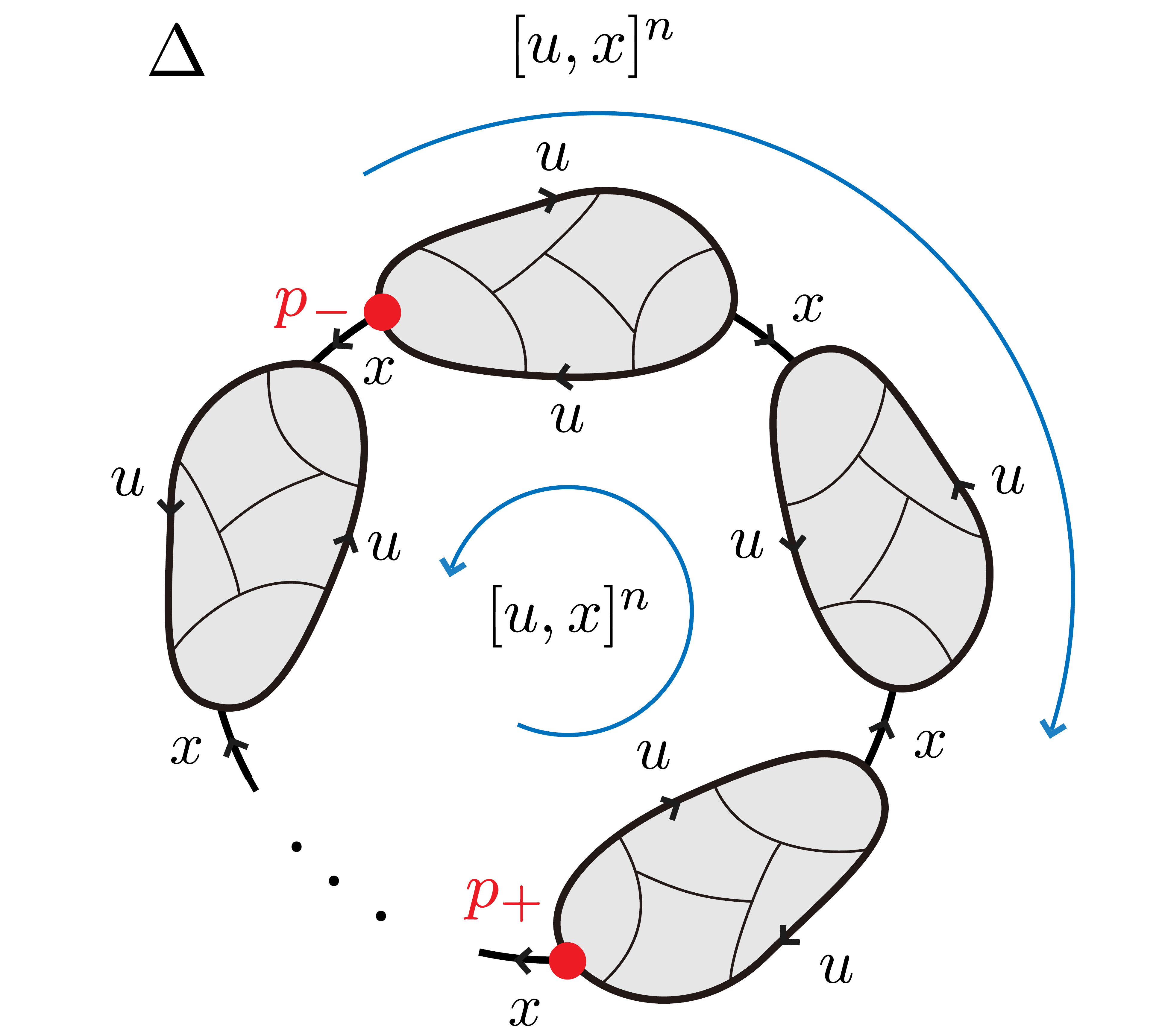}
\end{center}
\caption{The only possible shape of a minimal diagram $\Delta$ satisfying the hypothesis of Case~1
in the proof of Lemma~\ref{lem:unique_elementary_subgroup}}
\label{fig:annular_diagram_2}    		   	
\end{figure}

\medskip
\noindent {\bf Case~2.} $\Lab(\partial_\text{out} \Delta|_{p_{-}}) \equiv [u,x]^n$ and
$\Lab(\partial_\text{inn} \Delta|_{p_{+}}) \equiv [u,x]^{-n}$.

\medskip
In this case, if we read both the outer boundary $\partial_\text{out}\Delta$
and the inner boundary $\partial_\text{inn}\Delta$ starting from an edge labeled $x$
in the clockwise orientation,
they are both labeled as $(xu^{-1}x^{-1}u)^n$.
Since $\Delta$ is assumed to be minimal,
this yields that $\Delta$ is a trivial diagram,
meaning that $\Delta$ does not contain any $2$-cells.
Hence, if the path $p$ follows a clockwise direction,
$\Lab(p)$ starts with $u$ and ends with $x^{-1}$.
On the other hand, if the path $p$ follows a counterclockwise direction, $\Lab(p)$ starts with $x$ and ends with $u^{-1}$.
In either case, we have $f \equiv \Lab(p) \in \langle [u,x] \rangle$,
which implies that $E([u,x])=\langle [u,x] \rangle$, as desired.
\end{proof}

\begin{figure}[h]
\begin{center}
	\includegraphics[width=0.5\columnwidth]{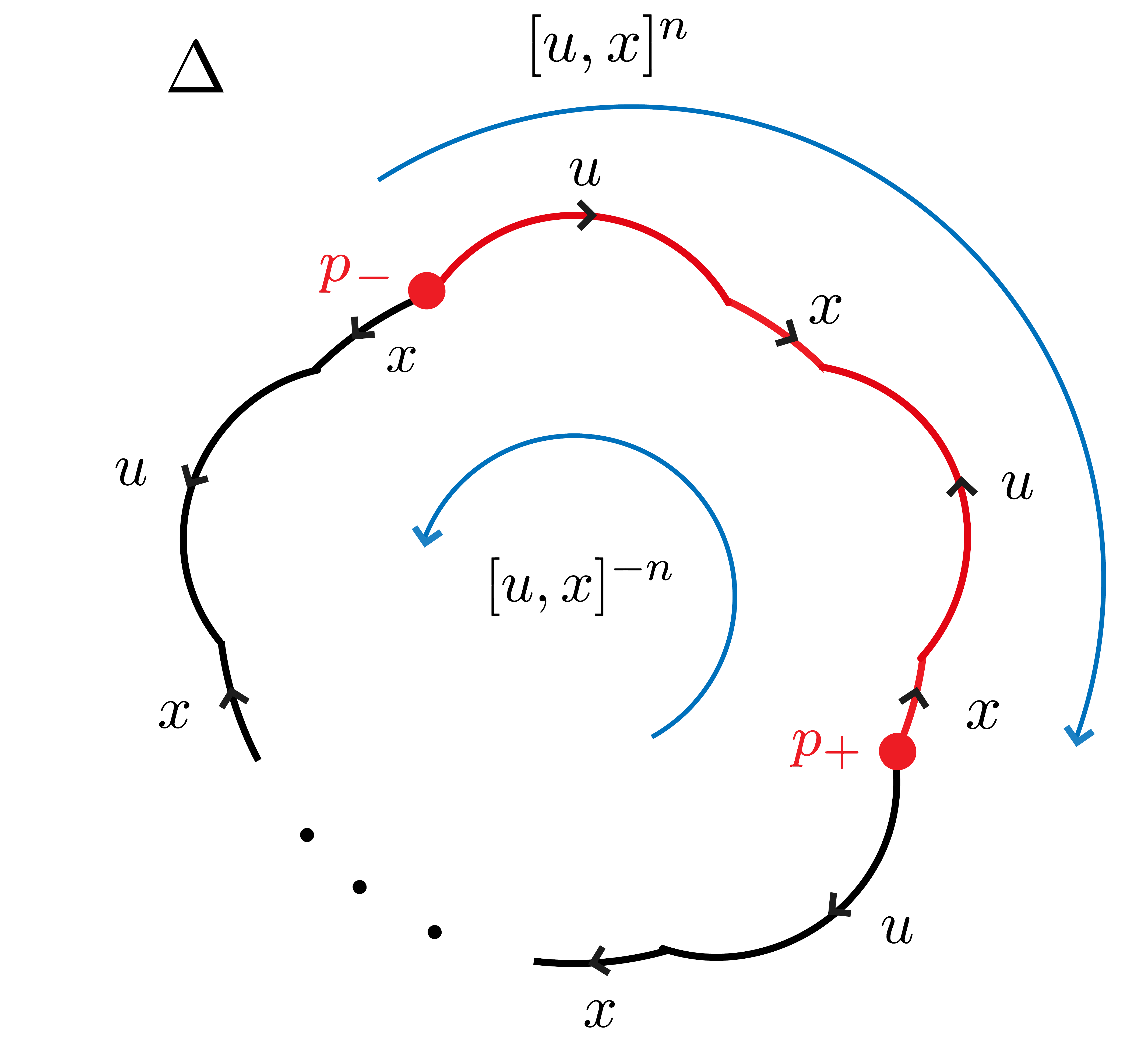}
\end{center}
\caption{The only possible shape of a minimal diagram $\Delta$ satisfying the hypothesis of Case~2
in the proof of Lemma~\ref{lem:unique_elementary_subgroup}}
\label{fig:annular_diagram_3}  	 	 	
\end{figure}

We are now ready to prove part (ii) of Theorem~\ref{thm:image_extension_theorem}.
Taking into consideration that $\mathbf{H} \ast \langle x \rangle$
is hyperbolic relative to $\mathbf{H}$, combining Lemma~\ref{lem:hyperbolic element} with Theorem~\ref{thm:hyperbolically_embedded} implies that
$E([u,x])$ is hyperbolically embedded into
$\mathbf{H} \ast \langle x \rangle$.
As a result, $\mathbf{H} \ast \langle x \rangle$
is hyperbolic relative to the collection of subgroups $\{\mathbf{H}, E([u,x])\}$.
By further combining this with the outcome of Lemma~\ref{lem:unique_elementary_subgroup},
we conclude that $\mathbf{H} \ast \langle x \rangle$ is hyperbolic relative to the collection of subgroups $\{\mathbf{H}, \langle [u,x] \rangle\}$.

At this point, recall that
$\mathbf{G}=\langle  \mathbf{H} \ast \langle x \rangle, t \svert t^{-1}vt= [u,x] \rangle$.
Since both $v \in \mathbf{H}$ and $[u,x]$ are elements of infinite order,
there exists a monomorphism $\zeta: \langle [u,x] \rangle \rightarrow \mathbf{H}$
such that $\zeta([u,x])=v$.
Consequently, according to Theorem~\ref{thm:Osin_combination_theorem},
we can deduce that $\mathbf{G}$ is hyperbolic relative to $\mathbf{H}$,
thereby completing the proof of part (ii) of Theorem~\ref{thm:image_extension_theorem}.

It remains to prove part (iii) of Theorem~\ref{thm:image_extension_theorem}.
Since $\tilde{\psi}$ is clearly an extension of $\psi$,
it is sufficient to show that $v \in \im \tilde{\psi}$.
Since $u \in \im \psi$ by the hypothesis of Theorem~\ref{thm:image_extension_theorem},
there exists $y \in \mathbf{H}$ such that $\psi(y)=u$.
Therefore, we have $\tilde{\psi}(t[y,x]t^{-1})=v$,
which implies that $v \in \text{im}\tilde{\psi}$, as desired.

\section{Examples}
\label{sec:proof_proposition}

In this section, we provide two examples
of a non-Hopfian relatively hyperbolic group with respect to a Hopfian subgroup
to demonstrate the application of Corollary~\ref{cor:image_extension_theorem}.
In both examples, we employ groups, as peripheral subgroups, that were previously established as Hopfian~\cite{Kim_Lee, Andreadakis}.

\begin{proposition}
\label{prop:first_example}
Let $\mathbf{H}$ be a group given by the presentation
\[
\begin{aligned}
\mathbf{H}=\langle b, c, s, k \svert
&b^2=c^9=1, \ b^{-1}cb=c^{-1},\\
&s^{-1}bs=bc^{-3}, \ s^{-1}cs=c,\\
&k^{-1}sk=s^3
\rangle,
\end{aligned}
\]
and suppose that $\mathbf{G}$ is a group given by the relative presentation
\[
\mathbf{G}=\langle \mathbf{H}, x, t \svert  t^{-1}\{(ks^{-1}k^{-1})b(ksk^{-1})cb^{-1}\}t= [c^3,x]  \rangle.
\]
Then $\mathbf{G}$ is a non-Hopfian relatively hyperbolic group with respect to a Hopfian subgroup $\mathbf{H}$.		
\end{proposition}

\begin{proof}
Let $u:=c^3$ and $v:=(ks^{-1}k^{-1})b(ksk^{-1})cb^{-1}$.
Then $u$ and $v$ satisfy the hypotheses of Theorem~\ref{thm:image_extension_theorem}.
The reason is as follows.
Note that the group $\mathbf{H}$
can be interpreted as an HNN-extension of the group
\[
\mathbf{H}_0=\langle b, c, s \svert b^2=c^9=1, \ b^{-1}cb=c^{-1}, \ s^{-1}bs=bc^{-3}, \ s^{-1}cs=c \rangle
\]
with stable letter $k$.
Then $u$ is a nontrivial element of $\mathbf{H}$ with $u^2=c^6 \neq 1$.
Also $v$ is a nontrivial element of $\mathbf{H}$ with infinite order by Britton's Lemma.
Let $\psi$ be an endomorphism of $\mathbf{H}$ defined by
\[
\psi(b)=b, \quad \psi(c)=c^3, \quad \psi(s)=s^3 \quad \text{and} \quad \psi(k)=k.
\]
Then it is easy to see that $u \in \ker \psi \cap \im \psi$ and $v \in \ker \psi \setminus \im \psi$.

Moreover, $\mathbf{H}$ is Hopfian by the result of \cite{Kim_Lee}.
Also $\tilde{\psi}$ is surjective,
since $b, k \in \im \psi \subseteq \im \tilde{\psi}$ by the definition of $\psi$, $c=\tilde{\psi}((k^2s^{-1}k^{-2})b^{-1}(k^2sk^{-2})t[c,x]t^{-1}b)$ and since $s=\tilde{\psi}(ksk^{-1})$.
Thus by Corollary~\ref{cor:image_extension_theorem}, the result follows.
\end{proof}

\begin{proposition}
\label{prop:second_example}
Let $\mathbf{H}$ be a group given by the presentation
\[
\mathbf{H}=\langle a,b,s \svert [a,b]=1,\ s^{-1}a^2s=a^4 \rangle,
\]
and suppose that $\mathbf{G}$ is a group given by the relative presentation
\[
\mathbf{G}=\langle \mathbf{H}, x, t \svert  t^{-1}(s^{-1}asa^{-2})t= [[sa^2s^{-1},b],x] \rangle.
\]	
Then $\mathbf{G}$ is a non-Hopfian relatively hyperbolic group with respect to a subgroup $\mathbf{H}$.
\end{proposition}

\begin{proof}
Let $u:=[sa^2s^{-1}, b^2]$ and $v:=s^{-1}asa^{-2}$.
Then $u$ and $v$ satisfy the hypotheses of Theorem~\ref{thm:image_extension_theorem}.
The reason is as follows.
Note that the group $\mathbf{H}$ given by the presentation
\[
\mathbf{H}=\langle a,b,s \svert [a,b]=1,\ s^{-1}a^2s=a^4 \rangle
\]
can be represented as an HNN-extension of the group
$\mathbf{H}_0=\langle a, b \svert [a, b]=1 \rangle$ with stable letter $s$.
Then by Britton's Lemma,
$u$ and $v$ are nontrivial elements of $\mathbf{H}$
such that $u^2 \neq 1$ in $\mathbf{H}$ and $v$ has infinite order.
Now, let $\psi$ be an endomorphism of $\mathbf{H}$ defined by
\[
\psi(a)=a^2, \quad \psi(b)=b \quad \text{and} \quad \psi(s)=s.
\]
Then it is straightforward to verify that
$u \in \ker \psi \cap \im \psi$ and $v \in \ker \psi \setminus \im \psi$.

Moreover, $\mathbf{H}$ is Hopfian
due to the result of \cite[Theorem~3]{Andreadakis} with $k=0$, $p=2$ and $q=4$.
Also, $\tilde{\psi}$ is surjective, since
$b, s \in \im \psi \subseteq \im \tilde{\psi}$ by the definition of $\psi$,
and since $a=\tilde{\psi}(st[[sas^{-1}, b], x]t^{-1}as^{-1})$.
Thus, the result follows by Corollary~\ref{cor:image_extension_theorem}.
\end{proof}

\end{document}